\documentclass[12pt,titlepage]{amsart}
\usepackage{amsfonts, amssymb, amsmath,amsthm, mathdots, verbatim, pmat}
\newtheorem{theorem}{Theorem}[section]
\newtheorem{proposition}[theorem]{Proposition}
\newtheorem{lemma}[theorem]{Lemma}

\theoremstyle{remark}
\newtheorem{remark}[theorem]{Remark}
\newtheorem{remarks}[theorem]{Remarks}

\theoremstyle{definition}

\DeclareMathOperator{\GL}{{\mathrm{GL}}}
\DeclareMathOperator{\SL}{{\mathrm{SL}}}

\DeclareMathOperator{\Hom}{{\mathrm{Hom}}}
\DeclareMathOperator{\End}{{\mathrm{End}}}
\DeclareMathOperator{\Ext}{{\mathrm{Ext}}}

\DeclareMathOperator{\Real}{{\mathrm{Re}}}
\DeclareMathOperator{\diag}{{\mathrm {diag}}} 	
\newcommand{\abs}[1]{|#1|}
\newcommand{\N}{{\mathbf N}}
\newcommand{\Z}{{\mathbf Z}}

\newcommand{\R}{{\mathbf R}}
\newcommand{\C}{{\mathbf C}}
\newcommand{\Ds}{{\mathcal {D}\{s\}}}
\newcommand{\Dss}{{\mathcal {D}[[s]]}}
\newcommand{\DS}{{\mathcal {DS}}}
\newcommand{\DR}{{\mathcal {DR}}}
\newcommand{\A}{{\mathcal A}}

\newcommand{\CG}{{{\mathbf C}G}}

\newcommand{\tH}{{\tilde H}}
\newcommand{\tJ}{{\tilde J}}
\newcommand{\tR}{{\tilde R}}
\newcommand{\tS}{{\tilde S}}

\newcommand{\tW}{{\tilde W}}
\newcommand{\U}{{\mathcal U}}
\newcommand{\Uhat}{{\widehat{\mathcal U}}}
\begin{document}
\title{The Divisor Matrix, Dirichlet Series and $\SL(2,\mathbf Z)$}
\author{Peter Sin}
\address{Department of Mathematics, University of Florida,
PO Box 118105,
Gainesville, FL 32611--8105, USA}
\email{sin@math.ufl.edu}
\author{John G. Thompson}
\address{Department of Mathematics, University of Florida,
PO Box 118105, Gainesville, FL 32611--8105, USA}
\email{jthompso@math.ufl.edu}
\date{}
\begin{abstract}
A representation of $\SL(2,\Z)$ by integer matrices acting
on the space of analytic ordinary Dirichlet series
is constructed, in which the standard unipotent element acts as multiplication by the Riemann zeta function. It is then  shown that the Dirichlet series in the orbit of the zeta function are related to it by algebraic equations.
\end{abstract}
\maketitle
\section {Introduction}
This paper is concerned with group actions on the space 
of analytic Dirichlet series.
A \emph{formal} Dirichlet series is a series of the form
$\sum_{n=1}^\infty a_nn^{-s}$, where $\{a_n\}_{n=1}^\infty$
is a sequence of complex numbers  and $s$ is formal variable. 
Such series form an algebra $\Dss$ under the operations of termwise addition and scalar multiplication and multiplication defined by Dirichlet convolution:
\begin{equation}
(\sum_{n=1}^\infty a_nn^{-s})(\sum_{n=1}^\infty b_nn^{-s})=
\sum_{n=1}^\infty(\sum_{ij=n}a_ib_j)n^{-s}.
\end{equation}
 It is well known that this algebra, sometimes called
the {\it algebra of arithmetic functions}, is isomorphic with the algebra 
of formal power series in  a countably infinite set of variables.
The  \emph{analytic} Dirichlet series, those which converge for some 
complex value of the variable $s$, form a subalgebra $\Ds$, 
shown in \cite{BM} to be a local,
non-noetherian unique factorization domain. As a vector space we can
identify $\Dss$ with the space of sequences and $\Ds$ with the 
subspace  of sequences satifying a certain polynomial growth condition. 
It is important for our purposes that the space of sequences is
the dual $E^*$ of a space $E$ of countable dimension, since the 
linear operators on $\Ds$ of interest to us are induced from operators on $E$. 
We take $E$ to be the space of columns, indexed by positive integers, which have only finitely many nonzero entries. In the standard basis of $E$, endomorphisms acting on the left are represented by column-finite matrices 
with rows and columns indexed by the positive integers. They act on $E^*$
by right multiplication. The  endomorphisms of $E$ which preserve $\Ds$ 
in their right action form a subalgebra $\DR$. There is a natural
embedding of $\Ds$ into  $\DR$  mapping a series to the {\it
multiplication  operator} defined by convolution with the series.
The Riemann zeta function $\zeta(s)=\sum_{n=1}^\infty n^{-s}$, 
$\Real(s)>1$, is mapped to the \emph{divisor matrix} 
$D=(d_{i,j})_{i,j\in\N}$, defined by
\begin{equation*}
d_{i,j}=\begin{cases}1,\quad \text{if $i$ divides $j$,}\\
                    0\quad \text{otherwise.}\end{cases}
\end{equation*}

It is also of interest to study noncommutative subalgebras of
$\DR$ or nonabelian subgroups of $\DR^\times$ which contain $D$.
In \cite{Th}, it was shown that $\langle D\rangle$ could be embedded
as the cyclic subgroup of index $2$ in an infinite dihedral subgroup 
of $\DR^\times$.
Given a group $G$, the problem of finding a subgroup of  $\DR^\times$ 
isomorphic with $G$ and containing $D$
is equivalent to the problem of finding a matrix
representation of $G$ into $\DR^\times$ in which
some group element is represented by $D$.

As a reduction step for this general problem, it is
desirable to transform the divisor matrix into a Jordan canonical form.
Since $\DR$ is neither closed under matrix inversion nor similarity,
so such a reduction is useful only if the transition matrices
belong to $\DR^\times$. We show explicitly
(Lemma~\ref{Zproperties} and  Theorem~\ref{Zinverse} ) that $D$ 
can be transformed to a Jordan canonical form by matrices
in $\DR^\times$ with integer entries.

The remainder of the paper is devoted to the group $G=\SL(2,\Z)$.
We consider the problem of constructing a representation 
$\rho:G\to\DR^\times$ such that the standard unipotent 
element $T=\left(\begin{smallmatrix}
1&1\\0&1\end{smallmatrix}\right)$ is represented by $D$
and such that an element of order $3$ acts without fixed points.
The precise statement of the solution of this problem is 
Theorem~\ref{main} below.
Roughly speaking, the result on the Jordan canonical form reduces the
problem to one of constructing a representation of $G$ in which $T$
is represented by a standard infinite Jordan block. In the simplified
problem the polynomial growth condition on the matrices representing group 
elements becomes a certain exponential growth condition
and the fixed-point-free condition is unchanged. The construction is
the content of Theorem ~\ref{JTmodule}. 
Since the group $G$ has few relations it is relatively easy to define
matrix representations satisfying the growth and fixed-point-free conditions
and  with $T$ acting indecomposably. However, the growth condition is not in 
general preserved under similarity and it is a more delicate 
matter to find such a representation for which one can prove that the matrix representing $T$ can be put into Jordan form by transformations preserving the growth condition.

As we have indicated, the construction of the
representation $\rho$ involves
making careful choices and we do not know of an abstract
characterization of $\rho$ as a matrix representation. 
The $\C G$-module affording $\rho$
can, however, be characterized as the direct sum of
isomorphic indecomposable modules where the isomorphism  type
of the summands is uniquely determined up to $\C G$-isomorphism 
by the indecomposable action of $T$ and the existence of
a filtration by standard $2$-dimensional modules (Theorem~\ref{uniqueness}).
Although $\rho$ is just one among many 
matrix representations of $G$ into $\DR^\times$ which satisfy
our conditions, we proceed to examine the orbit 
of $\zeta(s)$  under $\rho(G)$. We find (Theorem~\ref{cubicthm}) 
that one series $\varphi(s)$ in this orbit is related to $\zeta(s)$ 
by the cubic equation
\begin{equation}
(\zeta(s)-1)\varphi(s)^2+\zeta(s)\varphi(s)-\zeta(s)(\zeta(s)-1)=0.
\end{equation}
We also show (Theorem~\ref{orbit}) that, as a consequence
of relations in the image of the group algebra, the other series in the orbit
belong to $\C(\zeta(s),\varphi(s))$.

The cubic equation may be rewritten as
\begin{equation}
-\varphi(s)=(\zeta(s)-1)(\varphi(s)^2+\varphi(s)-\zeta(s)).
\end{equation}
The second factor on the right is a unit in $\Ds$, so
$\varphi$ and $\zeta(s)-1$ are associate irreducible elements in the
factorial ring $\Ds$.
The fact that there is a cubic equation relating $\zeta(s)$
and an associate of $\zeta(s)-1$ should be contrasted with
the classical theorem of Ostrowski \cite{O}, which states that $\zeta(s)$
does not satisfy any algebraic differential-difference equation. 

Matrices resembling finite truncations of the divisor matrix were studied 
by Redheffer in \cite{Re}. For each natural number $n$ he considered
the matrix obtained from the upper left  $n\times n$ submatrix of
$D$ by setting each entry in  the first column equal to $1$. 
Research on Redheffer's matrices has been motivated  by the fact that their
determinants are the values of Mertens' function, which links them directly
to the Riemann Hypothesis. (See \cite{BJ}, \cite{V1} and \cite{V2}.)


\section{Basic definitions and notation}
Let $\N$ denote the natural numbers $\{1,2,\ldots\}$ 
and $\C$ the complex numbers.
Let $E$ be the free $\C$-module with basis $\{e_n\}_{n\in\N}$.
With respect to this basis the endomorphism ring $\End_\C(E)$, 
acting on the left of $E$, becomes identified with the ring $\A$
of matrices $A=(a_{i,j})_{i,j\in\N}$, with complex entries,
such that each column has only finitely many nonzero entries.
The dual space $E^*$ becomes identified with the space $\C^\N$
of sequences of complex numbers, with $f\in E^*$
corresponding to the sequence $(f(e_n))_{n\in\N}$.
We will write  $(f(e_n))_{n\in\N}$ as $f$ and  $f(e_n)$ as $f(n)$  for short.

In this notation the natural right action of $\End_{\C}(E)$
on $E^*$ is expressed as a right action of $\A$ on $\C^\N$ by
\begin{equation*}
(fA)(n)=\sum_{m\in\N}a_{m,n}f(m), \qquad f\in\C^\N, A\in\A.
\end{equation*}
(The sum has only finitely many nonzero terms.)

Let $\DS$ be the subspace of $\C^\N$ consisting of the sequences $f$ 
for which there exist positive constants $C$ and $c$ such that
for all $n$, $\abs{f(n)}\leq Cn^c$. 
A sequence $f$ lies in $\DS$ if and only if the Dirichlet series 
$\sum_nf(n)n^{-s}$ converges for some complex number $s$,
which gives a  canonical bijection between $\DS$ and the space $\Ds$ 
of analytic Dirichlet series.
Let  $\DR$ be the subalgebra of $\A$ consisting of all elements 
which leave $\DS$ invariant.

A sufficient condition for membership in
$\DR$ is provided by the following lemma, whose proof is straightforward.
\begin{lemma}\label{ivanish} Let $A=(a_{i,j})_{i,j\in\N}\in \A$. Suppose
that there exist positive constants $C$ and $c$ such that
the following hold.
\begin{enumerate}
\item[(i)] $a_{i,j}=0$ whenever $i>Cj^c$.
\item[(ii)] For all $i$ and $j$ we have $\abs{a_{i,j}}\leq Cj^c$.
\end{enumerate}
Then $A\in\DR$. Furthermore, the set of all elements of $\A$
which satisfy these conditions, where the constants may
depend on the matrix, is a subring of $\DR$. 
\end{lemma}
We let $\DR_0$ denote the subring of $\DR$ defined by the lemma.
If $\sum_{n\in N}f(n)n^{-s}\in\Ds$, then its multiplication
operator has matrix with $(i,ni)$ entries equal to $f(n)$ for
all $i$ and $n$ and all other entries zero, so the
multiplication operators form commutative subalgebra of $\DR_0$.

\begin{remarks}\label{DRexample}
There exist elements of $\DR$ 
which do not satisfy the hypotheses of Lemma~\ref{ivanish}.
An example is the matrix 
$(a_{i,j})_{i,j\in\N}$ defined by
\begin{equation*}
a_{i,j}=\begin{cases}\frac{1}{j^{j^2}},\quad\text{if $i=j^j$,}\\
                     0,\quad\text{otherwise.}\\
\end{cases}
\end{equation*}

There are invertible matrices in $\DR_0$  whose inverses 
are not in $\DR$.
For example, the matrix $(b_{i,j})_{i,j\in\N}$, given by 
\begin{equation*}
b_{i,j}=\begin{cases}0,\quad\text{if $i>j$,}\\
                      1,\quad\text{if $i=j$,}\\
                      -1 \quad\text{if $i< j$}
\end{cases}
\end{equation*}
is obviously in $\DR_0$, while its inverse, given by
\begin{equation*}
b'_{i,j}=\begin{cases}0,\quad\text{if $i>j$,}\\
                      1,\quad\text{if $i=j$,}\\
                     2^{j-i-1} \quad\text{if $i< j$}
\end{cases}
\end{equation*}
is not in $\DR$.

\end{remarks}

\section{An action of $\SL(2,\Z)$ on Dirichlet series}
The group $G=\SL(2,\mathbf Z)$ is generated by the matrices
\begin{equation}\label{gens}
S=\begin{bmatrix}0&-1\\
          1&0\end{bmatrix},
\quad\text{and}\quad
T=\begin{bmatrix}1&1\\
          0 &1\end{bmatrix}.
\end{equation}
These matrices satisfy the relations
\begin{equation}\label{rels}
S^4=(ST)^6=1,\quad S^2=(ST)^3
\end{equation}
which, as is well known, form a set of defining relations
for $\SL(2,\Z)$ as an abstract group.

With the above definitions, we can state one of our principal 
results.

\begin{theorem}\label{main} There exists a representation 
$\rho:\SL(2,\Z)\rightarrow \A^\times$ with the following properties.
\begin{enumerate}
\item[(a)] The underlying $\C\SL(2,\Z)$-module $E$ has an ascending filtration
$$
0=E_0\subset E_1\subset E_2\subset\cdots
$$
of $\C\SL(2,\Z)$-submodules such that for each $i\in\N$,
the quotient module $E_{i}/E_{i-1}$ 
is isomorphic to the standard $2$-dimensional $\C\SL(2,\Z)$-module.
\item[(b)] $\rho(T)=D$.
\item[(c)] $\rho(Y)$ is an integer matrix for every $Y\in\SL(2,\Z)$.
\item[(d)] $\rho(\SL(2,\Z))\subseteq\DR_0$.
\end{enumerate}
\end{theorem}

The facts needed for the proof of Theorem~\ref{main} are
established in the following sections and the proof is completed in
Section~\ref{proofmain}.

\section{A Jordan form of the divisor matrix}\label{JCF}

For $m$, $k\in\N$, let
\begin{equation*}
A_k(m)=\lbrace (m_1,m_2,\ldots,m_k)\in(\N\setminus\{1\})^k\mid\, m_1m_2\cdots m_k=m
\rbrace
\end{equation*}
and let $\alpha_k(m)=\abs{A_k(m)}$. 

The following properties of these numbers follow from the definitions.
\begin{lemma}\label{alphaproperties}
\begin{enumerate}
\item[(a)] $\alpha_k(1)=0$.
\item[(b)] $\alpha_k(m)=0$ if $m<2^k$ and $\alpha_k(2^k)=1$.
\item[(c)] 
\begin{equation*}
(\zeta(s)-1)^k=\sum_{m=2^k}^\infty\frac{\alpha_k(m)}{m^s}.
\end{equation*}
\end{enumerate}
\end{lemma}

By considering the first $k-1$ entries of elements of
$A_k(m)$, we see that for  $k>1$, we have
\begin{equation}\label{count} 
\alpha_k(m)=\left(\sum_{d\mid m}\alpha_{k-1}(d)\right) - \alpha_{k-1}(m).
\end{equation}
Induction yields the following formula.
\begin{lemma}\label{alphacount}
\begin{equation}\sum_{i=1}^{k-1}(-1)^{k-1-i}
\sum_{d\mid m}\alpha_i(d)=\alpha_k(m)+(-1)^{k}\alpha_1(m).
\end{equation}
\end{lemma}

\begin{lemma}\label{bound} There exists a constant $c$ such that $\alpha_k(m)\leq m^c$
for all $k$ and $m$.
\end{lemma}
\begin{proof} We choose $c$ with $\zeta(c)=2$. We proceed
by induction on $k$. Since $\alpha_1(m)\leq 1$, the result is true
when $k=1$.
Suppose for some $k$ we have that for all $m$,
\begin{equation*}
\alpha_k(m)\leq m^c.
\end{equation*}
Then by (\ref{count}) we have
\begin{equation*}
\alpha_{k+1}(m)\leq\sum_{\begin{smallmatrix} d\mid m\\
                                             1<d<m\end{smallmatrix}}
\alpha_k(m/d)\leq m^c
\sum_{\begin{smallmatrix} d\mid m\\
             1<d<m\end{smallmatrix}}d^{-c}\leq m^c(\zeta(c)-1)=m^c,
\end{equation*}
which completes the inductive proof.
\end{proof}
\begin{remark}Since $\zeta(2)=\pi^2/6$,
the constant $c$ can be chosen from the real interval $(1,2)$.
\end{remark}

Let $J=(J_{i,j})_{i,j\in\N}$ be the matrix defined by
\begin{equation*}
J_{i,j}=\begin{cases}1,\quad \text{if $j\in\{i,2i\}$},\\
                    0\quad \text{otherwise}.\end{cases}
\end{equation*}

Let $Z=(\alpha(i,j))_{i,j\in\N}$ be the matrix described in the following way.
The odd rows have a single nonzero entry, equal to $1$ on
the diagonal. Let $i=2^kd$ with $d$ odd. Then the $i^\mathrm{th}$ row of $Z$ is 
equal to the $d^\mathrm{th}$ row of $(D-I)^{k}$.

\begin{lemma}\label{Zproperties} The matrix $Z$ has the following properties:
\begin{enumerate}
\item[(a)] $\alpha(i,j)=\delta_{i,j}$, if $i$ is odd.
\item[(b)] If $i=d2^k$, where $d$ is odd and $k\geq 1$, then
$$
\alpha(i,j)=\begin{cases}\alpha_k(j/d)\quad\text{if $d\mid j$,}\\
                             0 \quad\text{otherwise.}
             \end{cases}
$$
\item[(c)] $\alpha(im,jm)=\alpha(i,j)$ whenever $m$ is odd.
\item[(d)] $Z$ is upper unitriangular.
\item[(e)] $ZDZ^{-1}=J$.
\item[(f)] $Z\in\DR_0$.
\end{enumerate}
Moreover, $Z$ is the unique matrix satisfying (a) and (e).
\end{lemma}
\begin{proof}
Part (a) is by definition.
Part (b) follows from Lemma~\ref{alphaproperties}(c) upon multiplying
by the Dirichlet series with one term $d^{-s}$.
Part (c) is a special case of (b).
Part (d) also follows from (b).
Part (f) is then immediate from Lemma~\ref{bound}.
In the equation 
\begin{equation}\label{ZDJ}
Z(D-I)=(J-I)Z
\end{equation}
the $n^\mathrm{th}$ row of each side is equal 
to the $2n^\mathrm{th}$ row of $Z$.
This proves (e) since $Z$ is invertible by (d).
To prove the last statement we see that if (e) holds then by (\ref{ZDJ})
we have for all $i$ and $k\in\N$,
\begin{equation*}
\sum_{\begin{smallmatrix}j\mid k\\
                        j<k\end{smallmatrix}} \alpha(i,j)=\alpha(2i,k),
\end{equation*}
which determines $Z$ uniquely since the rows with odd index are 
specified by (a).
\end{proof}

Our aim is to determine $Z^{-1}$ explicitly.

For each prime $p$ and each integer $m$ let $v_p(m)$ denote
the exponent of the highest power of $p$ which divides $m$
and let $v(m)=\sum_pv_p(m)$.

\begin{lemma}\label{signchar}  We have for all $m\in\N$,
\begin{equation*}
\sum_{k=1}^{v(m)}(-1)^k
\alpha_k(m)=\begin{cases} (-1)^{v(m)},\quad \text{if $m$ is squarefree and $m>1$},\\
                      0\quad \text{otherwise.}\end{cases}
\end{equation*}     
\end{lemma}

\begin{proof}
The case $m=1$ is trivial.
Suppose $m=p_1^{\lambda_1}p_2^{\lambda_2}\cdots p_r^{\lambda_r}$, with 
$\lambda_1$, $\lambda_2$,\dots, $\lambda_r\geq 1$ and $r\geq 1$.
In the ring of formal power series $\C[[t_1,\ldots,t_r]]$ 
in $r$ indeterminates we set
\begin{equation}
\begin{aligned}
y&=\frac{1}{(1-t_1)(1-t_2)\cdots(1-t_r)}-1\\
  &=\sum_{(n_1,\ldots,n_r)\in(\N\cup\{0\})^r}\alpha_1(p_1^{n_1}\cdots p_r^{n_r})t_1^{n_1}\cdots t_r^{n_r}.
\end{aligned}
\end{equation}
Then for $k\geq1$
\begin{equation}
y^k=\sum_{(n_1,\ldots,n_r)\in(\N\cup\{0\})^r}
\alpha_k(p_1^{n_1}\cdots p_r^{n_r})t_1^{n_1}\cdots t_r^{n_r}.
\end{equation}
Then we have
\begin{equation}
\begin{aligned}
\prod_{i=1}^r(1-t_i)-1&=\frac{-y}{1+y}\\
                      &=\sum_{k\in\N}(-1)^ky^k\\
                      &=\sum_{(n_1,\ldots,n_r)\in(\N\cup\{0\})^r}[\sum_{k\in\N}(-1)^k\alpha_k(p_1^{n_1}\cdots p_r^{n_r})]t_1^{n_1}\cdots t_r^{n_r}.
\end{aligned}
\end{equation}
The lemma follows by equating the coefficients of monomials.
\end{proof}
\begin{remark}
The above proof is similar to the argument in \cite{L}, p.21, used to show that $\sum_k(-1)^k\alpha_k(m)/k$ is equal to $1/h$ if $m$ is the $h$-th power of a prime, and zero otherwise. The lemma has also the following enumerative proof,
based on  another combinatorial interpretation of the
sets $A_k(m)$. From the above factorization of $m$
let $\lambda$ be the partition of $n$ defined by the 
$\lambda_i$. Let $N=\{1,\ldots, n\}$ and let
 $F_\lambda$ be the set of functions
$h:N\rightarrow\{p_1,\ldots,p_r\}$ such that
$\abs{h^{-1}(p_i)}=\lambda_i$ for $i=1$,\dots,$r$.
The symmetric group $S_n$ acts transitively on the right of $F_\lambda$ 
by the rule $(h\sigma)(y)=h(\sigma(y))$, $y\in N$, $\sigma\in S_n$.
The stabilizer $S_\lambda$ of the function mapping the first $\lambda_1$ 
elements to $p_1$, the next $\lambda_2$ elements to $p_2$, etc. is
isomorphic to $S_{\lambda_1}\times S_{\lambda_2}\times\cdots\times S_{\lambda_r}$.
A $k$-decomposition of $n$ is a $k$-tuple $(n_1,\ldots,n_k)$
of integers $n_i\geq 1$ such that $n_1+n_2+\cdots+n_k=n$.
Let $\Pi=\{\sigma_1,\ldots\sigma_{n-1}\}$ be the set of fundamental
reflections, with $\sigma_i=(i,i+1)$. The subgroup $W_K$ of $S_n$
generated by a subset $K$ of $\Pi$ is called a 
standard parabolic subgroup of rank $\abs K$. 
Given a $k$-decomposition $(n_1,\ldots,n_k)$ of $n$,
we have a set decomposition of $N$ into subsets
$N_1=\{1,\ldots, n_1\}$, $N_2=\{n_1+1,\ldots, n_1+n_2\}$, \dots, $N_k=\{n_1+\cdots+n_{k-1}+1,\ldots, n\}$. The stabilizer of this decomposition is
a standard parabolic subgroup of rank $n-k$ and this correspondence
is a bijection between $k$-decompositions and standard parabolic subgroups
of rank $n-k$.

For each pair $((n_1,\ldots,n_k), h)$ consisting of a $k$-decomposition
and a function $h\in F_\lambda$, we obtain an element
$(m_1,\ldots, m_k)\in A_k(m)$ by setting
$m_i=\prod_{j\in N_i}h(j)$. Every element of $A_k(m)$
arises in this way and two pairs define the same element
of $A_k(m)$ if and only if the $k$-decompositions are
equal and the corresponding functions are in the same orbit
under the action of the parabolic subgroup of the $k$-decomposition.

Thus, we have
\begin{equation*}
\alpha_k(m)=\abs{A_k(m)}=\sum_{\begin{smallmatrix}K\subseteq\Pi\\
\abs{K}=n-k\end{smallmatrix}}\abs{\{\text{$W_K$-orbits on $F_\lambda$}\}}.
\end{equation*}
The number of $W_K$-orbits on $F_\lambda$
can be expressed as the inner product of permutation characters, so
\begin{equation*}
\alpha_k(m)=\sum_{\begin{smallmatrix}K\subseteq\Pi\\
\abs{K}=n-k\end{smallmatrix}}\langle 1_{W_K}^{S_n}, 1_{S_\lambda}^{S_n}\rangle.
\end{equation*}
Now, it is a well known fact \cite{So} that
\begin{equation*}
\sum_{K\subseteq\Pi}(-1)^{\abs{K}}1_{W_K}^{S_n}=\epsilon,
\end{equation*}
where $\epsilon$ is the sign character. Hence,
\begin{equation*}
\begin{aligned}
\sum_{k=1}^n(-1)^k\alpha_k(m)
&=(-1)^n\langle \sum_{K\subseteq\Pi}(-1)^{\abs K}1_{W_K}^{S_n}, 1_{S_\lambda}^{S_n}\rangle\\
&=(-1)^n\langle \epsilon, 1_{S_\lambda}^{S_n}\rangle\\
&=(-1)^n\langle \epsilon, 1\rangle_{S_\lambda}\\
&=\begin{cases} (-1)^n,\quad\text{if $\lambda=1^n$},\\
                   0,\quad\text{otherwise.}\end{cases}
\end{aligned}
\end{equation*}
\end{remark}

Let $X$ be the diagonal matrix with $(i,i)$ entry equal to
$(-1)^{v_2(i)}$, for $i\in\N$.
\begin{theorem}\label{Zinverse} $Z^{-1}=XZX$. In particular, $Z^{-1}\in\DR_0$. 
\end{theorem}
\begin{proof}
If $i$ is odd then the $i^\mathrm{th}$ row of $Z$ is zero except for $1$
in the $i^\mathrm{th}$ column, so the same holds for $XZX$.
By the last assertion of Lemma~\ref{Zproperties} it is sufficient 
to show that
\begin{equation*} D(XZX)=(XZX)J
\end{equation*}
or, equivalently,
\begin{equation*} (XDX)Z=Z(XJX).
\end{equation*}
The matrices $XDX=(d'_{i,j})_{i,j\in\N}$ and $XJX=(c'_{i,j})_{i,j\in\N}$ 
are given by
\begin{equation*}
d'_{i,j}=\begin{cases}(-1)^{v_2(i)+v_2(j)},\quad \text{if $i\mid j$},\\
                    0\quad \text{otherwise}.\end{cases}, \qquad
c'_{i,j}=\begin{cases}1,\quad \text{if $j=i$},\\
                     -1,\quad \text{if $j=2i$},\\
                     0\quad \text{otherwise}.\end{cases}
\end{equation*}
 Thus, we must show that
\begin{equation}\label{recurrence}
\sum _{m\geq 1}(-1)^{v_2(m)}\alpha(im,j)=
\begin{cases}\alpha(i,j),\quad\text{if $j$ is odd,}\\
\alpha(i,j)-\alpha(i,j/2),\quad\text{if $j$ is even.}
\end{cases}
\end{equation}
It is sufficient to consider the case $i=2^k$, for $k\geq1$,
by Lemma~\ref{Zproperties}(c). In this case, 
the left hand side of (\ref{recurrence}) can be rewritten as
\begin{equation}\label{summation}
\begin{aligned}
\sum_{\begin{smallmatrix}d\mid j\\
                         \text{$d$ odd}\end{smallmatrix}}&
\sum_{e=0}^{v(j)-k-v(d)}(-1)^e\alpha(2^{k+e},j/d)\\
=(-1)^k\sum_{\begin{smallmatrix}d\mid j\\
                         \text{$d$ odd}\end{smallmatrix}}&
\sum_{r=k}^{v(j/d)}(-1)^r\alpha(2^r,j/d).\end{aligned}
\end{equation}

Suppose that we can prove for all $j$, that
\begin{equation}\label{claim}
(-1)^k\sum_{d\mid j}\sum_{r=k}^{v(j/d)}(-1)^r\alpha(2^r,j/d)=\alpha(2^k, j).
\end{equation}
Then we will have proved (\ref{recurrence}) if $j$ is odd.
If $j$ is even, we note that $d$ is a divisor of $j/2$
if and only if $2d$ is an even divisor of $j$, so that (\ref{claim})
implies
\begin{equation*}
\begin{aligned}
\alpha(2^k,j/2)
&=(-1)^k \sum_{d\mid (j/2)}
\sum_{r=k}^{v((j/2)/d)}(-1)^r\alpha(2^r,(j/2)/d)\\
&=(-1)^k\sum_{\begin{smallmatrix}d'\mid j\\
                         \text{$d'$ even}\end{smallmatrix}}
\sum_{r=k}^{v(j/d')}(-1)^r\alpha(2^r, j/d').
\end{aligned}
\end{equation*}
Thus, from (\ref{summation}) we see that
(\ref{recurrence}) also follows from (\ref{claim}) when $j$ is even.
It remains to prove (\ref{claim}). We can assume $j>1$, 
by Lemma~\ref{alphaproperties}(a).
Lemma~\ref{signchar}, applied to the left hand side of (\ref{claim}), yields
\begin{equation}\label{simpler}
(-1)^{k-1} +(-1)^{k-1}\sum_{d\mid j}\left(\sum_{r=1}^{k-1}(-1)^r\alpha_r(j/d)\right)
\end{equation}
because the total contribution from the squarefree case of
Lemma~\ref{signchar} is
\begin{equation*}
(-1)^{k}\sum_{\begin{smallmatrix}d\mid j\\
                         \text{$j/d$ squarefree}\\
                          j/d >1\end{smallmatrix}}(-1)^{v(j/d)}=(-1)^{k-1}.
\end{equation*}
We can rewrite (\ref{simpler}) as
\begin{equation*}
 (-1)^{k-1}+\sum_{r=1}^{k-1}(-1)^{k-1-r}\sum_{d\mid j}\alpha_r(d),
\end{equation*}
which, by Lemma~\ref{alphacount} is equal to $\alpha_k(j)$.
This proves (\ref{claim}). 
\end{proof}

\section{Construction of representations}\label{constr}


Let $J_\infty$ be the ``infinite Jordan block'', indexed by $\N\times\N$,
defined by
\begin{equation*}
(J_\infty)_{i,j}=\begin{cases}1,\quad\text{if $j=i$ or  $j=i+1$,}\\
0 \quad\text{otherwise.}
\end{cases}
\end{equation*}

In the following theorem,  $T$ and $S$ are the generators of
$\SL(2,\Z)$ defined in (\ref{gens}).
\begin{theorem}\label{JTmodule} There exists a representation 
$\tau:\SL(2,Z)\rightarrow \A^\times$ with the following properties.
\begin{enumerate}
\item[(a)] Let $E_i$ be the subspace of $E$ spanned by
$\{e_1,\ldots,e_{2i}\}$, $i\in\N$. Then 
$$
0=E_0\subset E_1\subset E_2\subset\cdots 
$$
is a filtration of $\C\SL(2,\Z)$-modules and for each $i\in\N$
the  quotient module  $E_{i}/E_{i-1}$ is isomorphic to the 
standard $2$-dimensional $\C\SL(2,\Z)$-module.
\item[(b)] $\tau(T)=J_\infty$.
\item[(c)] $\tau(Y)$ is an integer matrix for every $Y\in\SL(2,\Z)$.
\item[(d)] There is a constant $C$ such that for all $i$ and $j$
we have $\abs{\tau(S)_{i,j}}\leq 2^{Cj}$.
\end{enumerate}
\end{theorem}

Later we will show (Theorem~\ref{uniqueness}) 
that there is a unique $\C\SL(2,\Z)$-module
with a filtration by standard modules and such that $T$ acts indecomposably
and unipotently on every $T$-invariant subspace. 

We define a sequence of integers  $\{b_n\}_{n\geq 0}$ recursively  by\footnote{As J-P. Serre has pointed out to us, this is the sequence of Catalan numbers,
up to signs.}
\begin{equation}\label{bn}
b_0=b_1=1, \quad b_n+\sum_{\begin{smallmatrix}i,j\geq1\\
i+j=n\end{smallmatrix}}b_ib_j=0 \quad\text{for all $n\geq 2$.}
\end{equation}

Let $\C[[t]]$  denote the ring of formal power series over $\C$ and
let $g(t)\in\C[[t]]$ be defined by
\begin{equation*}
1+g(t)=\sum_{k=0}^\infty b_{k}t^{k}.
\end{equation*}
Then the recurrence relations satisfied by the $b_i$ can be stated as
the equation
\begin{equation}\label{quadratic}
g(t)^2+g(t)=t.
\end{equation}
Thus,
\begin{equation*}
g(t)=\frac{-1+\sqrt{1+4t}}{2},
\end{equation*}
where the positive square root is taken since $g(t)$ has no constant term.
By Taylor expansion we obtain
\begin{equation}\label{taylor}
b_{m}=\frac{(-1)^{m-1}}{m}\binom{2m-2}{m-1},\quad(m\geq 2),\quad  b_1=b_0=1.
\end{equation}


Let
\begin{equation*}
 B_0=T, \qquad B_1=\begin{bmatrix}0&1\\1&1\end{bmatrix},
 \qquad B_i=\begin{bmatrix}0&b_i\\b_i&0\end{bmatrix}, \quad(i\geq 2)
\end{equation*}

and define
\begin{gather*}
\tJ=\begin{pmat}[{||||}]
 B_0&B_1&B_2&B_3&\hdots\cr\-
 0&B_0&B_1&B_2&\hdots\cr\-
 0&0&B_0&B_1&\hdots\cr\-
 0&0&0&B_0&\hdots\cr\-
 \vdots&\vdots&\vdots&\vdots&\ddots\cr
 \end{pmat},\\
\tS=\diag(S,S,\ldots),\\
\tR=-\tS\tJ.
\end{gather*}

For any ring $R$ let $M_n(R)$ denote the ring of 
$n\times n$ matrices over $R$.
Let $\U$ denote the ring of matrices of the form
\begin{equation}\label{defU}
U=\begin{pmat}[{||||}]
 X^{(0)}&X^{(1)}&X^{(2)}&X^{(3)}&\hdots\cr\-
 0&X^{(0)}&X^{(1)}&X^{(2)}&\hdots\cr\-
 0&0&X^{(0)}&X^{(1)}&\hdots\cr\-
 0&0&0&X^{(0)}&\hdots\cr\-
 \vdots&\vdots&\vdots&\vdots&\ddots\cr
 \end{pmat},
\end{equation}
where, for all $n\geq 0$,
\begin{equation*}
X^{(n)}=\begin{bmatrix}x^{(n)}_{1,1}&x^{(n)}_{1,2}\\
                      x^{(n)}_{2,1}&x^{(n)}_{2,2}\end{bmatrix}\in M_2(\C)
\end{equation*}
and the $X^{(n)}$ are repeated down the diagonals. 
For example, $\tS$, $\tJ$ and
$\tR$ all belong to $\U$.
The center $Z(\U)$ consists of those matrices
in which the submatrices $X^{(n)}$ are all scalar matrices.
The map $\C[[t]]\to Z(\U)$ sending 
$\sum_{n\geq0} a_nt^n$ to the matrix with $X^{(n)}=a_{n}I$, for all $n\geq 0$,
is a $\C$-algebra isomorphism, and extends to a
$\C[[t]]$-algebra isomorphism 
\begin{equation}\label{gamma}
\gamma:\U\to M_2(\C[[t]]), \qquad U\mapsto \begin{bmatrix}x_{1,1}(t)&x_{1,2}(t)\\
                        x_{2,1}(t)&x_{2,2}(t) \end{bmatrix},
\end{equation}
where
\begin{equation*}
x_{i,j}(t)=\sum_{n=0}^\infty x^{(n)}_{i,j}t^n,
\qquad\text{$i$, $j\in\{1,2\}$.} 
\end{equation*}
We have:
\begin{equation}\label{Qtmats}
\gamma(\tS)=\begin{bmatrix}0&-1\\
               1& 0 \end{bmatrix}, \quad
\gamma(\tJ)=\begin{bmatrix}1& 1+g(t)\\
               g(t)& 1+t \end{bmatrix},
 \quad
\gamma(\tR)=\begin{bmatrix}g(t)& 1+t\\
                -1& -1-g(t) \end{bmatrix}.
\end{equation}

\begin{lemma}\label{relations} 
\begin{enumerate}
\item[(a)] $\tS^2=-I$.
\item[(b)] $\tR^2+\tR+I=0$.
\item[(c)] There exists a representation $\tau_1$ of $G$ such that
$\tau_1(S)=\tS$ and $\tau_1(T)=\tJ$.
\end{enumerate}
\end{lemma}
\begin{proof}
Part (a) is obvious and (b) is easy to check by direct computation
using (\ref{Qtmats}) and (\ref{quadratic}). 
By (a) and (b), the elements $\tS$ and $\tJ$ 
satisfy the defining relations (\ref{rels}) for $\SL(2,\Z)$,
so (c) holds.
\end{proof}

The representation $\tau_1$ satisfies all the conditions of
Theorem~\ref{JTmodule} except for (b). To complete the
proof of Theorem~\ref{JTmodule} we shall conjugate this
representation by an upper unitriangular integer matrix $P$ 
such that $P\tJ P^{-1}=J_\infty$. In order to check that
$P\tS P^{-1}$ satisfies condition (d) of Theorem~\ref{JTmodule}
we will need to compute $P$ and its inverse explicitly.

\subsection{Transforming $\tJ$ into Jordan form}\label{transP}
A matrix $P$ such that $P\tJ P^{-1}=J_\infty$  can be found
by following the usual method for computing Jordan blocks. Thus, for $n\in\N$,
we define the $n^\mathrm{th}$ row of $P$ to be the 
first row of $(\tJ-I)^{n-1}$, setting $(\tJ-I)^0=I$.
Then $P$ is upper unitriangular, hence invertible,
and, from its definition, $P$ satisfies the equivalent equation 
\begin{equation*}
P(\tJ-I)=(J_\infty-I)P.
\end{equation*}

We now compute the entries of $P$ explicitly. In order to do this,
we use the isomorphism $\gamma$ of (\ref{gamma}). 
Let $\tH=\gamma(\tJ-I)$. Then
\begin{equation*}
\tH=\begin{bmatrix}0&1+g(t)\\
                   g(t)&t\end{bmatrix}.
\end{equation*}
It easy to compute the powers of $\tJ-I$ by diagonalizing $\tH$. Since
$g(t)^2+g(t)=t$, the characteristic polynomial of $\tH$ is
$\chi(x)=x^2-tx-t$. Let $\lambda_1$ and $\lambda_2$ be the roots of this polynomial
in some extension field and for $n\geq 0$ let 
$h_n=(\lambda_1^{n+1}-\lambda_2^{n+1})/(\lambda_1-\lambda_2)$ 
be the complete symmetric polynomial of degree $n$ 
in two variables, evaluated at 
$(\lambda_1,\lambda_2)$. Then $h_n$ is a polynomial in the coefficients
of $\chi(x)$, so it is a polynomial in $t$. We have $h_0=1$ and $h_1=t$.
A straightforward computation shows that, for $n\geq 2$,
\begin{equation}\label{tHpower}
\tH^{n}=
\begin{bmatrix} th_{n-2}&(1+g(t))h_{n-1}\\
                g(t)h_{n-1}& h_{n}
\end{bmatrix}.
\end{equation}
It follows from (\ref{tHpower}) and the equation $g(t)^2+g(t)=t$ that the 
polynomials $h_n$ satisfy the recurrence
\begin{equation*}
h_n=th_{n-1}+th_{n-2},\quad(n\geq 2),\qquad h_0=1, \quad h_1=t.
\end{equation*}
By inspection, the solution is
\begin{equation*}
h_n=\sum_{r=0}^{\lfloor\frac{n}{2}\rfloor}\binom{n-r}{r}t^{n-r}.
\end{equation*}
Thus we can compute the entries of $P$ as coefficients of the powers of $t$ in 
the top rows of the $\tH^n$. For $\ell\geq 3$ and $s\geq 0$, we have
\begin{equation}\label{pentries}
\begin{aligned}
p_{\ell,2s+1}&=\text{coefficient of $t^{s}$ in $th_{\ell-3}$}\\
&=\binom{s-1}{\ell-s-2},\\
p_{\ell,2s+2}&=\text{coefficient of $t^{s}$ in $(1+g(t))h_{\ell-2}$}\\
&=\sum_{k=0}^{\lfloor s+1-\frac{\ell}{2}\rfloor}b_k\binom{s-k}
{\ell+k-s-2}.
\end{aligned}
\end{equation}
Here and elsewhere, we employ the convention for binomial coefficients
that $\binom{a}{b}=0$ unless $a\geq b\geq 0$.

We now turn to the computation of $P^{-1}$. Suppose a matrix
$Q=(q_{i,j})_{i,j\in\N}$ satisfies the two conditons
\begin{equation}\label{Pinv}
\tJ Q=Q J_\infty \qquad\text{and}\qquad q_{1,j}=\delta_{1,j}.
\end{equation}
The first conditon implies that $PQ$ commutes with $J_\infty$
and the second that $(PQ)_{1,j}=\delta_{1,j}$, from which
it follows that $PQ=I$ and $Q=P^{-1}$. We find a matrix
$Q$ satisfying (\ref{Pinv}) by first finding a matrix $A$ 
such that 
\begin{equation}\label{Acond}
(\tJ-I)A=A(J_\infty-I)
\end{equation}
and then modifying it. To compute  $A$ we must
first enlarge the ring $\U$.
Let $\Uhat$ denote the set of matrices of the form
\begin{equation*}
W=\begin{pmat}[{|||||}]
\hdots& X^{(m)}&X^{(m+1)}&X^{(m+2)}&X^{(m+3)}&\hdots\cr\-
\hdots&X^{(m-1)} &X^{(m)}&X^{(m+1)}&X^{(m+2)}&\hdots\cr\-
\hdots&X^{(m-2)}&X^{(m-1)}&X^{(m)}&X^{(m+1)}&\hdots\cr\-
\hdots&X^{(m-3)}&X^{(m-2)}&X^{(m-1)}&X^{(m)}&\hdots\cr\-
\ddots&\ddots&\ddots&\ddots&\ddots&\ddots\cr
 \end{pmat},
\end{equation*}
where the blocks $X^{(m)}\in M_2(\C)$, for $m\in\Z$,
are repeated down the diagonals and have
the property that for some $m_0\in\Z$, which may depend on $W$,
$X^{(m)}=0$ whenever $m<m_0$.  
We shall refer to the two columns of $W$ headed by $X^{(m)}$
as the $[m,1]$ column and the $[m,2]$ column, respectively.
For $W\in\Uhat$ and $m\in\Z$, we denote by $W(m)$
the submatrix of $W$ whose first column is the $[m,1]$ column of $W$.
To be concise, we can write $W=(X^{(m)})_{m\in\Z}$, since
the top row determines the whole matrix. 
A product is defined as follows. Let $W'=(Y^{(m)})_{m\in\Z}\in \Uhat$.
Then $WW'=(Z^{(m)})_{m\in\Z}$, where 
\begin{equation*}
Z^{(m)}=\sum_{i+j=m}X^{(i)}Y^{(j)}.
\end{equation*}
This product can be computed as an ordinary matrix product 
as follows. 
Let $m_0$ and $n_0$ be chosen such that
$X^{(m)}=0$ for all $m<m_0$ and $Y^{(n)}=0$ for all $n<n_0$.
Then $WW'$ is obtained from
the ordinary matrix product $W(m_0)W'(n_0)$
by adjoining columns of zeros to the left and 
declaring the first column of $W(m_0)W'(n_0)$ to be the $[m_0+n_0,1]$
column of the new matrix. 
The answer is independent of the choice of $m_0$ and $n_0$, due to the diagonal pattern of elements of $\Uhat$.
Together with the usual vector space structure on matrices,
the above product makes $\Uhat$ into a $\C$-algebra. 
The subset of elements $W\in\Uhat$ such that $X^{(m)}=0$ for all $m<0$
forms a subalgebra isomorphic to the algebra
$\U$ defined in (\ref{defU}). Let $\C((t))$ 
denote the field of formal Laurent series, the field of fractions
of $\C[[t]]$. The center $Z(\Uhat)$ consists of the elements 
in which all the submatrices $X^{(m)}$
are scalar. The map sending the Laurent 
series $\sum_n a_nt^n$ to the element 
$(X^{(m)})_{m\in\Z}$ such that $X^{(m)}=a_mI$ for all $m$, is an isomorphism
of $\C((t))$ with $Z(\Uhat)$. This extends to an
isomorphism of $\C((t))$-algebras 
\begin{equation*}
\widehat\gamma:\Uhat \to M_2(\C((t))),
\end{equation*}
which is the unique extension of the isomorphism (\ref{gamma}).

Now, the element $\tJ-I$ is invertible in $\Uhat$,
since $\tH=\widehat\gamma(\tJ-I)$ has determinant $-t$.
We define $A=(a_{i,j})_{i,j\in\N}$ by columns. For $n\in\N$,
we set the $n^\mathrm{th}$ column of $A$ equal to the $[0,1]$ column of
$(\tJ-1)^{-(n-1)}$.
Then $A$ satisfies (\ref{Acond}),  by construction. 
To compute the entries of $A$ we invert $\tH$ and its powers (\ref{tHpower})
to obtain
\begin{equation*}
(\tH)^{-1}=-t^{-1}\begin{bmatrix}t&-(1+g(t))\\
                                      -g(t)& 0\end{bmatrix}
\end{equation*}
and
\begin{equation*}
(\tH)^{-n}=(-1)^nt^{-n}\begin{bmatrix}h_n&-(1+g(t))h_{n-1}\\
                                      -g(t)h_{n-1}& th_{n-2}\end{bmatrix},
\qquad n\geq 2.
\end{equation*}
Then we read off the coefficients of the appropriate powers of $t$
in the first columns. The first two columns of $A$
are given by 
\begin{equation}\label{2cols}
\begin{gathered}
a_{i,1}=\delta_{i,1},\qquad i\in\N,\\
a_{1,2}=-1, \qquad a_{2,2}=1,\qquad a_{i,2}=0,\qquad i\geq 3.
\end{gathered}
\end{equation}
For $m\geq 3$ and $s\geq0$ we have
\begin{equation}\label{aentries}
\begin{aligned}
a_{2s+1,m}
&=\text{coefficient of $t^{-s}$ in $(-1)^{m-1}t^{-(m-1)}h_{m-1}$}\\
&=(-1)^{m-1}\binom{m-s-1}{s}\\
a_{2s+2,m}
&=\text{coefficient of $t^{-s}$ in $(-1)^mt^{-(m-1)}g(t)h_{m-2}$}\\
&=(-1)^m\sum_{k=1}^{\lfloor\frac{m}{2}\rfloor-s}b_k\binom{m-s-k-1}{s+k-1}.
\end{aligned}
\end{equation}
Let $Q=AJ_\infty=(q_{i,j})_{i,j\in\N}$. We check that $Q$
has the properties (\ref{Pinv}).
Since $\tJ A= A J_\infty$, it is clear that  $\tJ Q= Q J_\infty$.
We have
\begin{equation}\label{multj}
q_{i,j}=\begin{cases} a_{i,j}, \quad\text{if $j=1$,}\\
a_{i,j}+a_{i,j-1},\quad\text{if $j\geq 2$.}\end{cases}   
\end{equation}
Since $a_{1,m}=(-1)^{m-1}$, it follows that $q_{1,j}=\delta_{1,j}$.
Thus, $Q=P^{-1}$.

Finally, the entries of $Q$ are obtained by
applying (\ref{multj}) to (\ref{2cols}) and (\ref{aentries}).
Thus, $q_{i,1}=\delta_{i,1}$ and $q_{i,2}=\delta_{i,2}$,
for $i\in \N$. For $m\geq3$ and $s\geq 0$, we have
\begin{equation}\label{qentries}
\begin{aligned}
q_{2s+1,m}
&=(-1)^{m-1}\binom{m-s-2}{s-1}\\
q_{2s+2,m}
&=(-1)^m\sum_{k=1}^{\lfloor\frac{m}{2}\rfloor-s}b_k\binom{m-s-k-2}{s+k-2}.
\end{aligned}
\end{equation}

\begin{lemma}\label{APbounds}
For all $i$ and $j$ we have
$\abs{q_{i,j}}\leq 2^{3j}$ and $\abs{p_{i,j}}\leq 2^{2j}$.
\end{lemma}
\begin{proof}The bound   
$\abs{b_k}\leq 2^{2k-2}$ for $k\geq 1$ follows from (\ref{taylor}).
It is then elementary to verify the bounds of the lemma
from the formulae (\ref{pentries}) and (\ref{qentries}).
\end{proof}

\subsection*{Proof of Theorem~\ref{JTmodule} }
We define
\begin{equation*}
\tau(Y)=P\tau_1(Y)P^{-1}, \qquad Y\in\SL(2,\Z).
\end{equation*}
From its construction, $\tau$ satisfies conditions (a), (b) and (c)
of Theorem~\ref{JTmodule}. It follows from Lemma~\ref{APbounds} 
that $\tau(S)=P\tS P^{-1}$ satisfies (d).
\qed

\section{Proof of Theorem~\ref{main}}\label{proofmain}
The matrix $Z^{-1}$ studied in Section~\ref{JCF} is the transition matrix from the
basis $\{e_n\}_{n\in\N}$ of $E$ to a new basis $\{e'_n\}_{n\in\N}$.
The linear transformation represented by the divisor matrix $D$ in the basis
$\{e_n\}_{n\in\N}$  is represented by $J$  in the basis $\{e'_n\}_{n\in\N}$.

Since $\N=\bigcup_{\text{$d$ odd}}\{ d2^{k-1}\mid k\in\N\}$
we have a decomposition

\begin{equation*}
E=\bigoplus_{\text{$d$ odd}}E(d),
\end{equation*}
Where $E(d)$ is the subspace of $E$ spanned by the elements $e'_{d2^{k-1}}$,
$k\in\N$.

We consider the isomorphisms
\begin{equation*}    
\phi_d:E\rightarrow E(d), \quad e_k\mapsto e'_{d2^{k-1}}.
\end{equation*}

For each odd number $d$ let 
$\A(d)$ be the subring of $\A$ consisting of matrices
whose entries $a_{i,j}$ are zero unless $i$ and $j$
both belong to the set $\{d2^{k-1}\mid k\in\N\}$.

The above isomorphisms induce isomorphisms
\begin{equation*}
\psi_d:\A\rightarrow \A(d).
\end{equation*}
and a homomorphism
\begin{equation*}
\psi:\A\rightarrow\prod_{\text{$d$ odd}}\A(d)\subseteq \A,\qquad \psi(A)=(\psi_d(A))_{\text{$d$ odd}}
\end{equation*}
We have 
\begin{equation*}
\psi(J_{\infty})=J.
\end{equation*}
Now for $A\in\A$, we have $\psi(A)_{i,j}=0$ unless there exists an odd number $d$
and $k$, $\ell\in\N$ with $(i,j)=(d2^{k-1},d2^{\ell-1})$, in which case
$\psi(A)_{i,j}=A_{k,\ell}$.

Let $\tau$ be the representation given by Theorem~\ref{JTmodule}
and let $\tau(S)=(s_{k,\ell})_{k,\ell\in\N}$.
By Theorem~\ref{JTmodule}(a), $s_{k,\ell}=0$ if $k>\ell+1$.
This means $\psi(\tau(S))_{i,j}=0$ if $i>2j$.
By Theorem~\ref{JTmodule}(d), there exists a constant
$C$ such that 
$\abs{s_{k,\ell}}\leq 2^{C\ell}$, for all $k$ and $\ell$, which
implies that $\abs{\psi(\tau(S))_{i,j}}\leq 2^C j^C$, for
all $i$ and $j$. We conclude that  $\psi(\tau(S))\in\DR_0$. Since 
$\psi(\tau(T))=J$, it follows that $\psi(\tau(\SL(2,\Z)))\subseteq\DR_0$.
Finally, the representation 
\begin{equation*}
\rho:\SL(2,\Z)\to\A^\times, \qquad Y\mapsto Z^{-1}\psi(\tau(Y))Z
\end{equation*}
satisfies all of the conditions of Theorem~\ref{main}.
The proof of  Theorem~\ref{main} is now complete. \qed

\begin{remarks}By a  closer examination of the proof we can
strengthen the conclusions of Theorem~\ref{main} in the following ways.
First, we have actually constructed the subgroup of $\DR_0^\times$
isomorphic to the direct product of copies $\SL(2,\Z)$ (indexed by the
odd numbers)  with the representation $\rho$ conjugate to the 
diagonal embedding.
Also part (d) can be sharpened to state that for $Y\in\SL(2,\Z)$
we have $\rho(Y)_{i,j}=0$ whenever $i>2j$.
\end{remarks}

\section{Extending representations  to $\GL(2,\Z)$} 
Let 
\begin{equation*}
W=\begin{pmatrix}0&1\\1&0\end{pmatrix}.
\end{equation*}
We have
\begin{equation}\label{wrels}
W^2=1,\quad WSW=S^{-1},\quad WRW=R^{-1}.
\end{equation}
The relations (\ref{wrels}) and (\ref{rels}) together
form a set of defining relations for  $\GL(2,\Z)=\langle\SL(2,\Z), W\rangle$.

In the following lemma the isomorphism $\gamma$
is defined in (\ref{gamma}) and the matrices $\gamma(\tS)$,
and $\gamma(\tR)$ are from  (\ref{Qtmats}). 

\begin{lemma}\label{wt}  Let
\begin{equation*}
W(t)=\frac{1}{\sqrt{t^2+4t+1}}\begin{bmatrix}-t&2g(t)+1\\
                                         2g(t)+1& t\end{bmatrix}.
\end{equation*}
Then $W(t)$ is, up to a sign, the  unique element of $\GL(2,\C[[t]])$
such that
\begin{enumerate}
\item[(i)]$W(t)^2=1$.
\item[(ii)]$W(t)\gamma(\tS)W(t)=\gamma(\tS)^{-1}$
\item[(iii)] $W(t)\gamma(\tR) W(t)=\gamma(\tR)^{-1}$
\item[(iv)]$W(0)=W$.
\end{enumerate}
\end{lemma}
\begin{proof}
The proof is straightforward, by matrix calculations in 
$M_2(\C[[t]])$, using the relation (\ref{quadratic}).
\end{proof}

\begin{lemma}\label{wbound} For $i$,$j\in\{1, 2\}$, let
$w_{i,j}(t)=\sum_{n=0}^\infty r_nt^n$. Then
there exists a constant $C$, such that $\abs{r_n}\leq 2^{Cn}$.
\end{lemma}
\begin{proof}
We consider those power series $\sum_{n=0}^\infty s_nt^n$ with real coefficients
for which there exists a constant $D$, which may depend on the series, such that $\abs{s_n}\leq 2^{Dn}$. We observe that the product of two such series has the same property.
Since $g(t)$ has this property and since $t^2+4t+1=(t+(2+\sqrt{3}))(t+(2-\sqrt{3}))$, we are reduced to proving the bound for the
Taylor series, centered at $0$, of $f(t)=(t+a)^{-\frac12}$, where $a>0$. 
We have
\begin{equation*}
f^{(n)}(t)=(-1)^n\frac{1\cdot3\cdot5\cdots(2n-1)}{2^n}(t+a)^{-\frac{(2n+1)}{2}},
\end{equation*}
Hence,
\begin{equation*}
\left|{\frac{f^{(n)}(0)}{n!}}\right|=\frac{1}{2^{2n}}\binom{2n}{n}a^{-\frac{(2n+1)}{2}}\leq
\frac{1}{\sqrt{a}}(\frac{1}{a})^n.
\end{equation*}
\end{proof}

\begin{proposition} The representation $\rho:\SL(2,\Z)\to\DR_0^\times$
can be extended to $\GL(2,\Z)$.
\end{proposition}
\begin{proof}
Set $\tW=\gamma^{-1}(W(t))$.
Then by Lemma~\ref{wt}, the group generated by $\tS$, $\tR$ and $\tW$ is
isomorphic to $\GL(2,\Z)$ and we can extend the representation
$\tau_1$ from $\SL(2,\Z)$ to $\GL(2,\Z)$ by setting $\tau_1(W)=\tW$.
Hence we can also extend the representations
$\tau$ and $\rho$ by setting $\tau(W)=P\tau_1(W)P^{-1}$
and $\rho(W)=Z^{-1}\psi(\tau(W))Z$.
 Then Lemma~\ref{wbound} and  Lemma~\ref{APbounds} imply that
$\rho(\GL(2,\Z))\subseteq\DR_0$.
\end{proof}

\begin{remark}
Note that $\tau_1(W)$, $\tau(W)$ and $\rho(W)$ are 
not integral matrices.
\end{remark}

\section{Uniqueness of $M_\infty$}
Let $S$ and $T$ be the generators of $G=\SL(2,\Z)$ 
as given in (\ref{gens}). 
Let $V$ denote the standard $2$-dimensional $\CG$-module.

We shall call a $\CG$-module \emph{$T$-indecomposable module} 
if $T$ acts indecomposably and unipotently
on every $T$-invariant subspace. One example is
the $\CG$-module, which we shall denote by $M_\infty$,
defined by the representation $\tau$ of Theorem~\ref{JTmodule}.

\begin{theorem}\label{uniqueness}  $M_\infty$ is the unique
$T$-indecomposable $\CG$-module 
which has an ascending filtration $\{M_n\}_{n\in\N}$
in which every quotient $M_n/M_{n-1}$ is isomorphic to $V$.
\end{theorem}
Some lemmas are needed for the proof of Theorem~\ref{uniqueness}.
\begin{lemma}\label{ext}$\Ext^1_{\CG}(V,V)\cong\C$.
\end{lemma}
\begin{proof} Suppose we have a module extension $M$ of $V$ by itself
and let $\mu:G\to\GL(M)$ denote the representation.
Since the cyclic group $\langle ST\rangle$ of order $6$ acts semisimply, we may 
choose a basis of $M$ such that
\begin{equation*}
\mu(ST)=\begin{bmatrix}ST&0\\
          0 &ST\end{bmatrix} 
\quad\text{and}\quad
\mu(S)=\begin{bmatrix}S& z(S)\\
          0&S\end{bmatrix},
\end{equation*}
for some $2\times2$ matrix $z(S)$.
Since $\mu(S)^2=-I$, we have
   $z(S)S+Sz(S)=0$, so
\begin{equation*}
z(S)=\begin{bmatrix}a&b\\
                    b&-a\end{bmatrix}
\end{equation*}
for some $a$, $b\in\C$.

By a further change of basis  we can  reduce to
\begin{equation*}
\mu(S)=\begin{bmatrix}0&-1&a&0\\
                      1&0&0&-a\\
                      0&0&0&-1\\
                      0&0&1&0\end{bmatrix},
\end{equation*}
while leaving $\mu(ST)$ unchanged. Thus, $\dim \Ext^1_{\CG}(V,V)\leq 1$.
Lastly, if $a\neq0$ then  $\mu(T)=-\mu(S)\mu(ST)$ acts indecomposably.
\end{proof}

\begin{lemma}\label{M(n)}For each natural number $n$ there is, 
up to isomorphism, a unique $T$-indecomposable $\CG$-module $M(n)$ 
of length $n$ and having all composition factors isomorphic to $V$ .
\end{lemma}
\begin{proof} We already have existence of such a module, as a submodule
of $M_\infty$.
We prove by induction that $\Ext^1_{\CG}(V,M(k))\cong\mathbf Q$.
The case $k=1$ is Lemma~\ref{ext}.
We apply $\Hom_{\CG}(V,-)$ to the short exact sequence
\begin{equation*}
0\to M(k-1)\to M(k)\to V\to 0.
\end{equation*}
The long exact sequence of cohomology is:
$$
\begin{aligned}
0\to &\Hom_{\CG}(V, M(k-1))\to \Hom_{\CG}(V, M(k))\to\Hom_{\CG}(V, V)\\
\to &\Ext^1_{\CG}(V, M(k-1))\to \Ext^1_{\CG}(V, M(k))\to\Ext^1_{\CG}(V, V)\to
\end{aligned}
$$
The desired conclusion follows by induction and Lemma~\ref{ext}.
\end{proof}

\begin{lemma}\label{extend} Let $M$, $M'$ be isomorphic to $M(n)$ and let
$N$, $N'$ be their maximal $\CG$-submodules. Then any $\CG$-isomorphism
from $N$ to $N'$ can be extended to an isomorphism from $M$ to $M'$.
\end{lemma}
\begin{proof}
We argue by induction on $n$, the case $n=1$
being trivial. We assume $n>1$. By Lemma~\ref{M(n)},
$N'$ has, for each $k\leq n-1$, a unique submodule  $N'(k)\cong M(k)$ 
of length $k$ and these are all the submodules of $N'$.
Let $\psi:N\to N'$ be a given isomorphism. Choose any isomorphism
$\phi:M\to M'$. Replacing $\phi$ by a scalar multiple, we can assume
that $\alpha:=\phi\vert_N-\psi\in\Hom_{\CG}(N,N')$ is not an isomorphism,
so it has a nonzero kernel $K$. Hence $\alpha$ induces
an isomorphism $N/K\to N'(k)$ for some $k<n-1$. By induction, this
isomorphism may be extended to an isomorphism $\overline\beta:M/K\to N'(k+1)$.
The induced map $\beta:M\to N'(k+1)$ is an extension of $\alpha$.
Thus, $\psi$ extends to $\phi-\beta$, which is an isomorphism,
since $N'(k+1)\subsetneq M'$.
\end{proof}

\subsection*{Proof of Theorem~\ref{uniqueness}}
Let $M_\infty$ and $M'_\infty$ be modules satisfying the conditions of Theorem~\ref{uniqueness}. Then the submodules $M_n$ and $M'_n$
in their respective filtrations are isomorphic with $M(n)$. 
By Lemma~\ref{extend} we can define
isomorphisms $\phi_n:M_n\to M'_n$ recursively for $n\in\N$, so that
$\phi_{n+1}$ extends $\phi_n$. We can therefore define
$\phi:M_\infty\to M'_\infty$ as follows. Each $m\in M_\infty$.
belongs to $M_n$ for some $n$. By the extension property, $\phi_n(m)$
does not depend on $n$, so we can define a map 
$\phi$ by $\phi(m)=\phi_n(m)$, which
is easily seen to be an isomorphism.\qed

\section{Dirichlet series in the $\SL(2,\Z)$-orbit of $\zeta(s)$}
We may identify $\DS$ with $\Ds$ and consider the 
action of $\SL(2,\Z)$ on analytic Dirichlet series via $\rho$.
We denote the Dirichlet series with one term $1^{-s}$ simply by $1$.
We have  $1.\rho(T)=\zeta(s)$.
We set $\varphi(s):=1.\rho(-S)$ and write
\begin{equation*}
\varphi(s):=\sum_{n=1}^\infty a_nn^{-s},
\end{equation*}
where $a_n=\rho(-S)_{1,n}$.
We denote the abscissae of conditional
and absolute convergence of $\varphi(s)$ by $\sigma_c$ and $\sigma_a$,
respectively.

Let $\C(\zeta(s), \varphi(s))$ 
be the subfield of the field of meromorphic functions
of the half-plane $\Real(s)>\max(1,\sigma_c)$
generated by the functions $\zeta(s)$ and $\varphi(s)$.  
It will be shown below that the Dirichlet series
in the orbit $1.\rho(\SL(2,\Z))$ all converge in this
half-plane and that the analytic functions they define 
belong to $\C(\zeta(s),\varphi(s))$.
Let $\Z G$ denote the integral group ring. The representation
$\rho$ extends uniquely to a ring homomorphism from
$\Z G$ to $\A$, which we will denote by $\rho$ also. The kernel
of this homomorphism contains the $2$-sided ideal $Q$ generated
by the elements $S+S^{-1}$ and $R+R^{-1}-1$. Since $R=ST$, we have 
the relation
\begin{equation*}
TS=1+ST^{-1}
\end{equation*}
in $\Z G/Q$. It follows that $\Z G/Q$ and hence $\rho(\Z G)$
is generated as an abelian group by the images of the elements $T^m$
and $ST^m$, $m\in\Z$.

\begin{theorem}\label{orbit} The Dirichlet series in the
common $\SL(2,\Z)$-orbit of $1$, $\zeta(s)$ and $\varphi(s)$
all converge for $\Real(s)>\max(1,\sigma_c)$, and belong to the additive
subgroup of $\C(\zeta(s),\varphi(s))$ generated by the elements
$\zeta(s)^m$ and $\varphi(s)\zeta(s)^m$, $m\in\Z$.
\end{theorem}
\begin{proof} We first note that $\zeta(s)$ has no zeros in
the half-plane $\Real(s)>\max(1,\sigma_c)$ and that 
$\frac{1}{\zeta(s)}=\sum_{n=1}^\infty\mu(n)n^{-s}$, converges absolutely there.
Here, $\mu(n)$ is the M\"obius function.
In this half-plane we have $1.\rho(T^m)=1.D^m=\zeta(s)^m$ and
$1.\rho(ST^m)=1.\rho(S)\rho(T^m)=-\varphi(s)\zeta(s)^m$, for every $m\in\Z$.  
The theorem now follows from the discussion preceding it.
\end{proof}

\section{The cubic equation relating $\zeta(s)$ and $\varphi(s)$}
Let $\N_0=\N\cup\{0\}$ be the set of nonnegative integers.
\begin{lemma}\label{rhos}
We have 
\begin{equation*}
a_n= \rho(-S)_{1,n}=\alpha_1(n)+\sum_{\ell\geq 4}(-1)^\ell\alpha_{\ell-1}(n)\sum_{k=2}^{\lfloor\frac{\ell}{2}\rfloor}b_k\binom{\ell-k-2}{k-2}.
\end{equation*}
(See  Section~\ref{JCF}  and formula (\ref{bn}) for the definitions of
$\alpha_k(n)$  and $b_k$.) 

\end{lemma}
\begin{proof}
This is computed directly from the general formula for $\rho$:
\begin{equation*}
\rho(-S)=Z^{-1}\psi(P\tau_1(-S)P^{-1})Z.
\end{equation*}
We recall the following information.
\begin{enumerate}
\item[(a)] The matrix $\tau_1(-S)$ is the block-diagonal matrix with the $2\times 2$ block $-S=\left(\begin{smallmatrix}0&1\\-1&0\end{smallmatrix}\right)$ repeated along the main diagonal. (Lemma~\ref{relations})
\item[(b)] The first rows of $Z^{-1}$ and $P$ are equal to the 
first row of the identity matrix. ( Lemma~\ref{Zproperties} and Section~\ref{transP}.)
\item[(c)] For $A=(a_{i,j})_{i,j\in\N}$, we have $\psi(A)_{i,j}=0$ 
unless there exist $k$, $\ell\in\N$ and an odd number $d$
such that $(i,j)=(2^{k-1}d,2^{\ell-1}d)$, in which case 
$\psi(A)_{i,j}=a_{k,\ell}$.(Section~\ref{proofmain}.)
\item[(d)] From formula (\ref{qentries}) the entries in the
second row of $P^{-1}=(q_{k,\ell})$ are given by $q_{2,1}=0$, $q_{2,2}=1$
and, for $\ell\geq 3$,
\begin{equation}
q_{2,\ell}=(-1)^\ell\sum_{k=2}^{\lfloor\frac{\ell}{2}\rfloor}b_k\binom{\ell-k-2}{k-2}.
\end{equation}
\item[(e)] The entries of the matrix  $Z=(\alpha(i,j))_{i,j\in\N}$ satisfy
the equation $\alpha(2^r,j)=\alpha_r(j)$, for $r\in\N$. 
(Lemma~\ref{Zproperties}.)
\end{enumerate}
By (b), the first row of $\rho(-S)$ is obtained by multiplying the first row
of $\psi(P\tau_1(-S)P^{-1})$ with $Z$. By (c), the only nonzero
entries in the first row  of $\psi(P\tau_1(-S)P^{-1})$ are the entries
$\psi(P\tau_1(-S)P^{-1})_{1,2^{\ell-1}}=(P\tau_1(-S)P^{-1})_{1,\ell}$,
for $\ell\in\N$. Then by (b) and (a),
\begin{equation}
\begin{aligned}
(P\tau_1(-S)P^{-1})_{1,\ell}=(\tau_1(-S)P^{-1})_{1,\ell}=q_{2,\ell}.
\end{aligned}
\end{equation}
Hence, by (e) and (d),
\begin{equation}
\begin{aligned}
a_n&=\sum_{\ell\in\N}q_{2,\ell}\alpha(2^{\ell-1},n)\\
   &=\alpha_1(n)+\sum_{\ell\geq 3}q_{2,\ell}\alpha_{\ell-1}(n),
\end{aligned}
\end{equation}
and the lemma follows since $q_{2,3}=0$, by (d).
\end{proof}

Let $\Omega=\{p_1,\ldots, p_r\}$ be a finite set primes and let $t_1$,\dots, $t_r$
be indeterminates. We will be interested in  the formal power series

\begin{equation*}
F_\Omega=\sum_{(n_1,\dots,n_r)\in\N_0^r}a_{{p_1}^{n_1}{p_2}^{n_2}\cdots {p_r}^{n_r}}t_1^{n_1}\cdots t_r^{n_r}.
\end{equation*}
Let
\begin{equation*}
y=\frac{1}{(1-t_1)(1-t_2)\cdots(1-t_r)}-1
  =\sum_{(n_1,\ldots,n_r)\in\N_0^r}\alpha_1(p_1^{n_1}\cdots p_r^{n_r})t_1^{n_1}\cdots t_r^{n_r}.
\end{equation*}
Then for $\ell\geq1$
\begin{equation*}
y^\ell=\sum_{(n_1,\ldots,n_r)\in\N_0^r}
\alpha_\ell(p_1^{n_1}\cdots p_r^{n_r})t_1^{n_1}\cdots t_r^{n_r}.
\end{equation*}
Then we have
\begin{equation}\label{longeq}
\begin{aligned}
\frac{-y}{1+y}&=\sum_{\ell\in\N}(-1)^\ell y^\ell\\
                      &=\sum_{(n_1,\ldots,n_r)\in\N_0^r}[\sum_{\ell\in\N}(-1)^\ell\alpha_\ell(p_1^{n_1}\cdots p_r^{n_r})]t_1^{n_1}\cdots t_r^{n_r}.
\end{aligned}
\end{equation}

Set
\begin{equation*}
\begin{aligned}
f_\Omega&=\sum_{(n_1,\dots,n_r)\in\N_0}\sum_{\ell\in\N}(-1)^\ell
\alpha_{\ell-1}({p_1}^{n_1}{p_2}^{n_2}\cdots {p_r}^{n_r})\sum_{k=2}^{\lfloor\frac{\ell}{2}\rfloor}b_k\binom{\ell-k-2}{k-2}t_1^{n_1}\cdots t_r^{n_r}\\
&=\sum_{\ell\in\N}\sum_{k=2}^{\lfloor\frac{\ell}{2}\rfloor}b_k(-1)^\ell
\binom{\ell-k-2}{k-2}y^{\ell-1}\\
&=\sum_{k\geq2}[\sum_{\ell\geq2k}(-1)^{\ell}\binom{\ell-k-2}{k-2}y^{\ell-1}]b_k.
\end{aligned}
\end{equation*}
By Lemma~\ref{rhos},
\begin{equation*}
F_\Omega=y+f_\Omega.
\end{equation*}
 For $k\geq 2$ we set
\begin{equation*}
C_k=\sum_{\ell\geq 2k}(-1)^\ell\binom{\ell-k-2}{k-2}y^{\ell-1}
\end{equation*}
so that 
\begin{equation*}
f_\Omega=\sum_{k\geq 2}b_kC_k.
\end{equation*}

Next we consider, for $k\in\N\setminus\{1\}$, the generalized 
binomial coefficients
\begin{equation*}
p_k(x)=\frac{(x-k-2)(x-k-3)\cdots (x-2k+1)}{(k-2)!}
\end{equation*}
as polynomials in $x$ of degree $k-2$. 
Note that $p_k(\ell)=\binom{\ell-k-2}{k-2}$ for $\ell$ an integer $\geq 2k$ 
but, for example, when $\ell-k-2$ is a negative integer, 
the value $p_k(\ell)$ may be nonzero,
whereas our convention concerning binomial
coefficients would say that $\binom{\ell-k-2}{k-2}=0$. 
In order to find $C_2$ and $C_3$, we shall evaluate
\begin{equation*}
\widehat{C_k}=\sum_{\ell\in\N}(-1)^\ell p_k(\ell)y^{\ell}.
\end{equation*}
For $k=2$, we have $p_2(\ell)=1$, so $\widehat{C_2}=\frac{-y}{1+y}$
by (\ref{longeq}). Hence 
\begin{equation}\label{C2}
C_2=\frac{-1}{1+y}-\sum_{\ell=1}^3(-1)^\ell y^{\ell-1}=
\frac{-1}{1+y}+1-y+y^2=\frac{y^3}{1+y}.
\end{equation}
For $k=3$, we have $p_3(\ell)=\ell-5$, so
\begin{equation}
\begin{aligned}
\widehat{C_3}&=\sum_{\ell\in\N}(-1)^\ell\ell y^\ell-5\sum_{\ell\in\N}(-1)^\ell y^\ell\\
&=(\frac{-y}{1+y}+\frac{y^2}{(1+y)^2})+5\frac{y}{1+y}\\
&=\frac{4y}{1+y}+\frac{y^2}{(1+y)^2},
\end{aligned}
\end{equation}
where the second equality is obtained by applying the operator 
$y\frac{d}{dy}$ to the first and second members of (\ref{longeq}).
Therefore,
\begin{equation}
\begin{aligned}
C_3&=\frac{4}{1+y}+\frac{y}{(1+y)^2}-[(-1)p_3(1)+p_3(2)y-p_3(3)y^2+p_3(4)y^3]\\
   &=\frac{4}{1+y}+\frac{y}{(1+y)^2}-4+3y-2y^2+y^3\\
   &=\frac{y^5}{(1+y)^2}.
\end{aligned}
\end{equation}

Suppose $k\geq3$. We have 
\begin{equation*}
\begin{aligned}
C_k&=y^{2k-1}+\sum_{\ell\geq2k+1}(-1)^\ell\binom{\ell-k-2}{k-2}y^{\ell-1}\\
&=y^{2k-1}+\sum_{\ell\geq2k+1}(-1)^\ell\binom{\ell-1-k-2}{k-2}y^{\ell-1}
+\sum_{\ell\geq2k+1}(-1)^\ell\binom{\ell-1-k-2}{k-3}y^{\ell-1}.
\end{aligned}
\end{equation*}
Set
\begin{equation*}
A=\sum_{\ell\geq2k+1}(-1)^\ell\binom{\ell-1-k-2}{k-2}y^{\ell-1},\qquad
B=\sum_{\ell\geq2k+1}(-1)^\ell\binom{\ell-1-k-2}{k-3}y^{\ell-1}.
\end{equation*}
In $A$, set $\ell'=\ell-1$ and in $B$, set $k'=k-1$. Then
\begin{equation*}
A=-yC_k, \qquad B=y^2C_{k-1}-y^{2k-1}
\end{equation*}
Thus, 
\begin{equation*}
C_k=y^{2k-1}+A+B=y^{2k-1}-yC_k+y^2C_{k-1}-y^{2k-1}=-yC_k+y^2C_{k-1}.
\end{equation*}
Therefore,
\begin{equation*}
C_k=\frac{y^2}{1+y}C_{k-1}, \qquad\text{with $C_2=\frac{y^3}{1+y}$}
\end{equation*}
so
\begin{equation*}
C_k=\frac{y^{2k-1}}{(1+y)^{k-1}}.
\end{equation*}
Hence
\begin{equation*}
\frac{C_kC_{k'}}{C_{k+k'}}=\frac{y^{2(k+k')-2}}{(1+y)^{k+k'-2}}.\frac{(1+y)^{k+k'-1}}{y^{2(k+k')-1}}=\frac{1+y}{y}.
\end{equation*}

\begin{equation*}
f_\Omega^2=(\sum_{k\geq 2}b_kC_k)^2=\sum_{k,k'\geq 2}b_kb_{k'}C_{k+k'}\frac{(1+y)}{y}
\end{equation*}
Then, from the definition  (\ref{bn}) of the $b_k$,
\begin{equation*}
\begin{aligned}
\frac{y}{1+y}f_\Omega^2&=\sum_{K\geq 4}(\sum_{k=2}^{K-2}b_kb_{K-k})C_K\\
                &=\sum_{K\geq 4}(-b_K-2b_{K-1})C_K\\
                &=-\sum_{K\geq 2}b_KC_K+b_2C_2 +b_3C_3-2\frac{y^2}{1+y}\sum_{L\geq3}b_LC_L\\
&=-f_\Omega-C_2+2C_3-\frac{2y^2}{1+y}f_\Omega-\frac{2y^2}{1+y}C_2\\
&=-(1+\frac{2y^2}{1+y})f_\Omega-\frac{y^3}{1+y}+\frac{2y^5}{(1+y)^2}
-\frac{2y^2}{1+y}.\frac{y^3}{1+y}.
\end{aligned}
\end{equation*}
Therefore, we have
\begin{equation}
yf_\Omega^2+(1+y+2y^2)f_\Omega+y^3=0.
\end{equation}
Since $F_\Omega=f_\Omega+y$, this yields
\begin{equation}\label{Pielliptic}
yF_\Omega^2+(1+y)F_\Omega-y(1+y)=0.
\end{equation}
Set
\begin{equation}\label{Pwz}
P(z,w)=zw^2+(1+z)w-z(1+z)
\end{equation}
The discriminant $\Delta(z)$ is equal to $(1+z)^2+4z^2(1+z)$.
Set $c=\min\{\abs{e}\mid \text{$e\in\C$ and $\Delta(e)=0$}\}$.
Then there is a formal power series $u=\sum_{n=0}^\infty\gamma_nz^n$ such that $P(z,u)=0$
and $u$ defines an analytic function in $\{z\in\C\mid \abs{z}<c\}$.
Now the roots of $\Delta(z)$ are $-1$ and $e$,$\overline{e}=
\frac{-1\pm\sqrt{-15}}{8}$.
Since $\abs{e}=\frac{1}{2}$, 
it follows that $u$ converges for $\abs{z}<\frac{1}{2}$.
Applied to (\ref{Pielliptic}), we see that if $t_i$ take complex values with
$\abs{\prod_{i=1}^r\frac{1}{1-t_i}-1}<\frac{1}{2}$, the power series
$F_\Omega$ converges. In particular for $s\in\C$ with sufficiently
large real part, we have convergence when we set the $t_i=p_i^{-s}$.
If we denote by $\N_\Omega$ the set of natural numbers for which
every prime factor belongs to $\Omega$, and define
\begin{equation}
\varphi_\Omega(s)=\sum_{n\in\N_\Omega}a_nn^{-s},\quad\text{and}\quad
\zeta_\Omega(s)=\sum_{n\in\N_\Omega}n^{-s},
\end{equation}
we obtain the equation
\begin{equation}\label{Omegacubic}
(\zeta_\Omega(s)-1)\varphi_\Omega(s)^2+\zeta_\Omega(s)\varphi_\Omega(s)-\zeta_\Omega(s)(\zeta_\Omega(s)-1)=0.
\end{equation}
Initially, we know that this equation holds for $s$ with sufficiently
large real part.
The Dirichlet series $\zeta_\Omega(s)$ and $\varphi_\Omega(s)$ converge
absolutely in the half-plane $\Real(s)>\max(1,\sigma_a)$, where both 
$\zeta(s)$ and $\varphi(s)$ converge absolutely. It is then a general property of Dirichlet series that they converge uniformly on compact subsets of this half-plane, defining analytic functions there. Then, by the principle
of analytic continuation, the equation (\ref{Omegacubic}) holds in this half-plane. If we take $\Omega$ to be the set of the first $r$ primes and allow $r$ to
increase, the resulting sequences of analytic functions $\zeta_\Omega(s)$ and 
$\varphi_\Omega(s)$ defined in the above half-plane converge to $\zeta(s)$
and $\varphi(s)$, respectively. 

\begin{theorem}\label{cubicthm} In the half plane $\Real(s)>\max(1,\sigma_c)$, we have
\begin{equation}\label{cubic}
(\zeta(s)-1)\varphi(s)^2+\zeta(s)\varphi(s)-\zeta(s)(\zeta(s)-1)=0.
\end{equation}
\end{theorem}
\begin{proof}
The validity of this algebraic relation for $\Real(s)>\max(1,\sigma_a)$
is immediate from the foregoing discussion. Since $\zeta(s)$ and $\varphi(s)$
represent analytic functions throughout the half-plane 
$\Real(s)>\max(1,\sigma_c)$ the relation is valid on this
larger region, by the principle of analytic continuation.
\end{proof}

\begin{remark}
Since $\phi(s)$ defines an analytic function  in $\Real(s)>\sigma_c$, it
follows that $\zeta(s)-1$ cannot be  equal to any root of $\Delta(z)$ for $s$
in this half-plane. By Theorem 11.6 (C) of \cite{Ti}, $\zeta(s)$
takes on every nonzero value in $\Real(s)>1$. Therefore, $\sigma_c>1$.
A sharper bound follows from  \cite{BoJ}, which proves
the existence of a constant $C\approx 1.764$
such that  the closure $M(\sigma)$ of the set
of values of $-\log\zeta(\sigma+it)$ , $t\in\R$, is
bounded  by a convex curve when $\sigma<C$,  
and a ring-shaped domain between two convex curves when $\sigma>C$.
From this it follows by computation that $\zeta(s)=\frac{7\pm\sqrt{-15}}{8}$ for some $s$ with $\Real(s)$ arbitrarily close to 1.8, so $\sigma_c\geq 1.8$. 
We also know from the results of \cite{Ti}, p.300, that $\zeta(s)$
never takes the value $\frac{-7\pm\sqrt{-15}}{8}$ when $\Real(s)>1.92$.
\end{remark}

\begin{remark} A slight modification of the discussion above shows
that (\ref{Omegacubic}) holds for an arbitrary set $\Omega$ of primes,
again for $\Real(s)>\sigma_a$.
\end{remark}

The function $\zeta(s)$ can be extended to a meromorphic
function in the whole complex plane, whose only singularity is a simple pole at 
$s=1$. Then equation (\ref{cubic}) defines analytic continuations 
of $\varphi(s)$ along arcs in the plane which do not
pass through $s=1$ or the branch points $\{s\mid \text{$\zeta(s)=0$ or 
$\frac{7\pm\sqrt{-15}}{8}$}\}$, with the exception that one of the two
branches at each point $s$ with $\zeta(s)=1$ has a simple pole there.
By \cite{BoJ}, we know that there is a constant
$C\approx 1.764$ such that $\zeta(s)\neq 1$ for all $s$ with $\Real(s)>C$.
\subsection{Some generalizations}
In the discussion following (\ref{Pwz}), we could equally well have
substituted $t_i=M(p_i)p_i^{-s}$,
where $M$ is any bounded, completely multiplicative
complex function of the natural numbers, such as a Dirichlet character.
In that case, if we set 
\begin{equation}
\zeta_M(s)=\sum_{n=1}^\infty M(n)n^{-s},\qquad \varphi_M(s)=\sum_{n=1}^\infty a_n M(n)n^{-s},
\end{equation}
the same reasoning shows that $\zeta_M(s)$ and $\varphi_M(s)$
are related by (\ref{cubic}), just as $\zeta(s)$ and $\varphi(s)$ are,
in a suitable half-plane.

We can also  extend our discussion to number fields. For this purpose,
a necessary remark is that, by Lemma~\ref{rhos}, the coefficient
$a_n$ in $\varphi(s)$ depends only on the partition 
$\lambda: e_1\geq e_2\geq\cdots e_r\geq 1$ defined by the exponents 
$e_i$ which occur in the prime factorization of $n$, in that if
$n$ and $n'$ define the same partition then $a_n=a_{n'}$. We write
$a_\lambda$ for this common value.

Let $K$ be a number field.
The factorization of an ideal $\mathfrak g$ of its ring of integers
$\mathfrak o$ into prime ideals determines a partition $\lambda$,
so we may we set $a_{\mathfrak g}=a_{\lambda}$.
With these notations, our previous discussion up to (\ref{Pwz}) 
remains valid if the set $\Omega$ is taken to be a finite set $\{\mathfrak P_1,\mathfrak P_2,\ldots,\mathfrak P_r\}$ of prime ideals in
$\mathfrak o$, instead of rational primes.
Then, in the paragraph following (\ref{Pwz}), if
we  substitute $t_i= N(\mathfrak P_i)^{-s}$, we deduce, as before, that the 
Dedekind zeta function of $K$,
\begin{equation}
\zeta_K(s)=\sum_{\mathfrak g} N(\mathfrak g)^{-s}
\end{equation}
is related to the Dirichlet series
\begin{equation} 
\varphi_{K}(s)=\sum_{\mathfrak g}
a_{\mathfrak g}N(\mathfrak g)^{-s}
\end{equation}
 by the cubic relation (\ref{cubic}), in the appropriate
half-plane.

\section{A functional equation for $\varphi(s)$}
The classical functional equation for $\zeta(s)$ can be written as
\begin{equation}
\zeta(1-s)=a(s)\zeta(s),
\end{equation}
where $a(s)=\frac{\Gamma(s/2)\pi^{-s/2}}
{\Gamma((1-s)/2)\pi^{-(1-s)/2}}$.

If we apply this to (\ref{cubic}) with $s$ replaced
by $(1-s)$ and then eliminate $\zeta(s)$ from the resulting
equation, using (\ref{cubic}), a functional equation
relating $\varphi(s)$ and $\varphi(1-s)$ is obtained.
Let
\begin{multline}
G(a,x,y)=a^4x^4-a^3x^2(x^2+x+1)(y^2+y+1)\\
+a^2[x^2(y^2+y+1)^2+y^2(x^2+x+1)^2-2x^2y^2]\\
-ay^2(x^2+x+1)(y^2+y+1)+y^4.
\end{multline}
Then $G(a,x,y)$ is irreducible in $\C[a,x,y]$ and 
 $G(a(s),\varphi(s),\varphi(1-s))=0$.


\subsection*{Acknowledgements}
We thank Peter Sarnak for some helpful discussions and for
bringing \cite{L} to our attention.

\end{document}